\documentclass[letterpaper]{ieeeconf}
\IEEEoverridecommandlockouts                              
\usepackage[utf8]{inputenc}
\usepackage{textcomp}

\usepackage{graphicx}
\usepackage{url}
\usepackage[table]{xcolor}
\usepackage{amsmath}
\usepackage[version=4]{mhchem}
\usepackage{siunitx}
\usepackage{graphicx}
\usepackage{longtable,tabularx}
\usepackage[style=ieee,doi=true]{biblatex}
\usepackage{tikz}
\usepackage{comment}
\usetikzlibrary{positioning, arrows.meta}
\AtBeginBibliography{\small}

\usepackage{amsmath} 
\usepackage{amsfonts}
\usepackage{amsthm}

\newtheorem{theorem}{Theorem}

\newtheorem{proposition}{Proposition}[section]

\theoremstyle{definition}

\newtheorem{remark}{Remark}

\usepackage{float} 
\usepackage[skip=0em,font=footnotesize]{caption}

\usepackage[margin = 1in]{geometry}

\usepackage{tikz}
\usetikzlibrary{shapes,arrows,calc,positioning}
\tikzstyle{bigblock} = [draw, fill=blue!20, rectangle, 
    minimum height=6em, minimum width=8em]
\tikzstyle{medblock} = [draw, fill=blue!20, rectangle, 
    minimum height=4em, minimum width=4em]    
\tikzstyle{mux} = [draw, fill=black!20, rectangle, 
    minimum height=5em, minimum width=0.1em]    
\tikzstyle{smallblock} = [draw, fill=blue!20, rectangle, 
    minimum height=2em, minimum width=3em]
    
\tikzstyle{data_block} = [draw, fill=green!20, rectangle, 
    minimum height=2em, minimum width=3em]
\tikzstyle{ops_block} = [draw, fill=blue!20, rectangle, 
    minimum height=2em, minimum width=3em]    
\tikzstyle{est_block} = [draw, fill=red!20, rectangle, 
    minimum height=2em, minimum width=3em]    
    
\tikzstyle{sum} = [draw, fill=blue!20, circle, node distance=1cm,minimum height=0.5cm]
\tikzstyle{signal} = [coordinate]
\tikzstyle{pinstyle} = [pin edge={to-,thin,black}]
\tikzstyle{block} = [draw, fill=blue!20, rectangle, 
    minimum height=3em, minimum width=9em]
\tikzstyle{blockS} = [draw, fill=blue!20, rectangle, 
    minimum height=3em, minimum width=4em]    
\tikzstyle{input} = [coordinate]
\tikzstyle{output} = [coordinate]
\usetikzlibrary{matrix}

\usetikzlibrary{positioning}
\usetikzlibrary{math}

\usepackage{hyperref}
\usepackage{xcolor}
\hypersetup{
    colorlinks,
    linkcolor={blue!100!black},
    citecolor={blue!50!black},
    urlcolor={blue!80!black}
}

\newcommand{\bc}{\begin{center}}
\newcommand{\ec}{\end{center}}
\newcommand{\benum}{\begin{enumerate}}
\newcommand{\eenum}{\end{enumerate}}
\newcommand{\nn}{\nonumber}
\newcommand{\matl}{\left[ \begin{array}}
\newcommand{\matr}{\end{array} \right]}

\renewcommand{\matl}{\begin{bmatrix}}
\renewcommand{\matr}{\end{bmatrix}}

\newcommand{\matls}{\left[ \begin{smallmatrix}}
\newcommand{\matrs}{\end{smallmatrix} \right]}
\newcommand{\isdef}{\stackrel{\triangle}{=}}

\newcommand{\rmI}{{\rm I}}

\newcommand{\rmd}{{\rm d}}

\newcommand{\rms}{{\rm s}}

\newcommand{\BBR}{{\mathbb R}}

\newcommand{\SA}{{\mathcal A}}

\newcommand{\SC}{{\mathcal C}}

\newcommand{\SF}{{\mathcal F}}

\newcommand{\SSS}{{\mathcal S}}

\title{Admissibility-Preserving Constraint Lifting for \\ Conditional Deadbeat Control of Discrete-Time Nonlinear Systems}
\author{
Jhon Manuel Portella Delgado
and
Ankit Goel
\thanks{Jhon Manuel Portella Delgado was a Ph.D. student in the Department of Mechanical Engineering, University of Maryland, Baltimore County, 1000 Hilltop Circle, Baltimore, MD 21250, when this work was performed. He is now a postdoctoral researcher with the Department of Aerospace Engineering, University of Michigan, Ann Arbor, MI 48109, USA. {\tt\small jportella@umbc.edu, jhonp@umich.edu}}%
\thanks{Ankit Goel is with the Department of Mechanical Engineering, University of Maryland, Baltimore County, 1000 Hilltop Circle, Baltimore, MD 21250. {\tt\small ankgoel@umbc.edu}}%
}

\begin{document}

\maketitle




\begin{abstract}
This paper presents a constraint-enforcing control framework for a class of discrete-time strict-feedback nonlinear systems.
The objective is to guarantee closed-loop stability while ensuring forward invariance of a prescribed safe set defined by state constraints.
The proposed approach transforms the constrained control problem into an equivalent unconstrained one through smooth constraint-lifting mappings constructed using strictly increasing sigmoid functions. 
Controller synthesis is then performed in the lifted coordinates, enabling recursive backstepping design while preserving the admissibility of the constrained states.
Conditions on the controller gains are derived to guarantee both asymptotic stability of the closed-loop system and forward invariance of the admissible domain of the lifting functions.
The analysis also establishes a conditional deadbeat property for the second backstepping step once the system trajectory enters a region in which the lifting-domain admissibility conditions are satisfied. 
Numerical simulations of a constrained Roup oscillator demonstrate the effectiveness of the proposed method in enforcing state constraints while tracking a reference command.

\end{abstract}
\textit{\bf keywords:} 
Constraint-lifting control, State-constrained control, Discrete-time nonlinear systems, Forward invariance, Deadbeat control, Backstepping control

\section{Introduction}
 
The enforcement of state constraints in discrete-time nonlinear systems is the problem of forward invariance of a prescribed safe set.
Safety-critical and constraint-enforcing controllers are typically designed in continuous time and discretized only for implementation, and this discretization step is not guaranteed to preserve the continuous-time invariance guarantees.
Guarantees established directly in discrete time are therefore needed for sampled-data systems, where safety, actuator limits, and admissibility conditions must be satisfied at every sampling instant.

Set-invariance methods provide the underlying theoretical framework for enforcing such constraints.
Constrained-control architectures offer a practical mechanism instead.
A widely used example is the reference governor, which augments a nominal controller with a supervisory law.
This supervisory law modifies the command whenever constraint violation is predicted \cite{Blanchini1999, Gilbert1995, Garone2017}.
These methods are effective and well-established, but they serve as add-on schemes rather than yielding an explicit stabilizing control law for the constrained plant itself.
More recently, barrier-based methods have also been extended to discrete-time systems through discrete control barrier functions, which enforce forward invariance by imposing a safety inequality at each time step \cite{Agrawal2017}.
However, these methods are typically implemented through an optimization problem, and in the discrete-time setting, the resulting program is not necessarily convex \cite{Agrawal2017}, which for nonlinear systems may increase computational complexity and complicate the stability analysis of the closed-loop system.
A third approach addresses constraints via nonlinear coordinate transformations that map a bounded state space to an unconstrained one, thereby enabling the use of classical nonlinear control tools on the transformed dynamics \cite{Guo2014}.
Although this idea is attractive, conventional transformation-based designs and barrier Lyapunov formulations may suffer from numerical ill-conditioning near the constraint boundary, where the transformed coordinates or their derivatives become large, and these difficulties become more pronounced when several constrained states are handled recursively \cite{Tee2009, Guo2014}.

Across these approaches, the gap is the same.
No method yields a control law that is explicit, free of online optimization, and well conditioned near the constraint boundary.
Motivated by this gap, this paper develops a controller for a class of discrete-time nonlinear systems.
The controller is designed using a constraint-lifting framework.
The proposed controller guarantees closed-loop stability, reference tracking, and satisfaction of the state constraints.

The central idea is to map constrained states into an unconstrained coordinate system via smooth constraint-lifting transformations constructed from strictly increasing sigmoid functions.
Controller synthesis is then carried out in the lifted coordinates using a recursive design, while the inverse transformation guarantees that the original state constraints are satisfied.
The contributions of this paper are as follows.
First, a constraint-lifting control framework is developed for discrete-time strict-feedback systems that transforms the constrained control problem into an equivalent unconstrained one.
Second, conditions on the controller gains are derived to guarantee both asymptotic stability of the closed-loop system and forward invariance of the admissible domain of the lifting transformations.
Third, the analysis establishes a conditional deadbeat property for the second backstepping step once the closed-loop trajectory enters a region in which the admissibility conditions are satisfied.
A preliminary version of this paper appeared as an arXiv preprint~\cite{delgado2026stateconstrained}.
 
The remainder of this paper is organized as follows.
Section~\ref{sec:Problem_Formulation} formulates the constrained tracking problem.
Section~\ref{sec:Constraint_Lifting} introduces the lifting transformation and the lifted dynamics.
Section~\ref{sec:Controller_design} develops the backstepping controller and its closed-loop error dynamics.
Section~\ref{sec:Admisibility_constraint_leifting} establishes admissibility and the conditional deadbeat property.
Section~\ref{sec:Numerical_Example} presents the numerical example, and Section~\ref{sec:conclusions} concludes the paper.

\section{Problem Formulation}
\label{sec:Problem_Formulation}
Consider the system
\begin{align}
    x_{k+1,1}
        &=
            f_1(x_{k,1})
            +
            g_1(x_{k,1})
            x_{k,2}
    \label{eq:x1_dot}
    \\
    x_{k+1,2}
        &=
            f_2(x_{k,1},x_{k,2})
            +
            g_2(x_{k,1},x_{k,2})
            u_{k},
    \label{eq:x2_dot}
\end{align}
where $x_{k,1},$ $x_{k,2} \in \BBR,$ are the states, $u_k \in \BBR$ is the control, and the functions $f_1 : \BBR \to \BBR,$ $g_1 : \BBR \to \mathbb{R},$ $f_2 : \BBR \times \BBR \to \BBR,$ and $g_2 : \BBR \times \BBR \to \mathbb{R}.$
The goal is to design a feedback controller $u_k$ such that $x_{k,1}$ tracks a desired command $x_{k,1}^{\rmd}$ while $|x_{k,1}| < \overline{x}_1,$ and $|x_{k,2}| < \overline{x}_2.$
In this work, we consider the scalar case to simplify the exposition and focus on the key ideas of the proposed constraint-lifting framework. 
The extension to multi-dimensional systems is conceptually similar and will be developed in future work.

\section{Constraint Lifting}
\label{sec:Constraint_Lifting}

The states to be constrained are \textit{lifted} using a constraint-lifting transformation to obtain an unconstrained representation in a lifted state space.
The sigmoid functions and their inverses to lift the constraints are described in detail in \cite{portella2026constrained}.

First, the states $(x_1,x_2)$ are mapped to intermediate variables $(\chi_1,\chi_2)$ such that the bounds on the components of $x_1$ and $x_2$ are normalized to $\pm1$. 
Consequently, whenever $(x_1,x_2)\in\SSS$, each component of $(\chi_1,\chi_2)$ lies in $(-1,1)$, that is, $(\chi_1,\chi_2)\in\SC$.
The normalized variables $(\chi_1,\chi_2)$ are then mapped to $(z_1,z_2)\in\BBR^2$ using a constraint-lifting function. 
Controller synthesis is performed in the lifted $z$-coordinates for the transformed state $x_{1}^{\rmd}$.
Finally, a transformation from $(z_1,z_2)$ to $(\zeta_1,\zeta_2)$ is introduced to simplify both the inverse mapping from $z$ to $x$ and the subsequent stability analysis.
The sequence of transformations from $x$ to $z$, together with the intermediate variables $\chi$ and $\zeta$, is shown in Figure \ref{fig:state_transformation}.

\begin{figure}[!ht]
\centering
    {%
    \begin{tikzpicture}[>={stealth'}, line width = 0.25mm]

    \node [input, name=ref]{};
    
    \node [smallblock, fill=green!20, rounded corners, right = 0.5cm of ref , 
           minimum height = 0.6cm, minimum width = 0.7cm] 
           (controller) {$\dfrac{x}{\overline{x}}$};
    
    \node [smallblock, rounded corners, right = 2cm of controller, 
           minimum height = 0.6cm , minimum width = 0.7cm] 
           (integrator)            {$\overline x\,\phi(\chi)$};
    
    \node [smallblock, fill=green!20, rounded corners, below = 0.5cm of integrator, 
           minimum height = 0.6cm , minimum width = 0.7cm] 
           (system) {$\dfrac{z}{\overline{x}}$};
           
    \node [smallblock, rounded corners, below = 0.5cm of controller, 
           minimum height = 0.6cm , minimum width = 0.7cm] 
           (integrator2)            {$\overline x\,\psi(\zeta)$};

    \node [output, right = 1.0cm of system] (output) {};
    
    
    \node [input, left = 1.0cm of controller] (reference) {};
    
    \draw [->] (controller) -- node [above] {$\chi\in \SC$} (integrator);
    \draw [->] (integrator.0) -- +(1,0) node [above, xshift = -1em] {$z\in\BBR^n$} |-  (system.0);
    \draw [->] (system.180)  -- node [above] {$\zeta \in \BBR^n$} (integrator2.0);
    \draw [->] (integrator2.180) -- +(-1,0) |- node [above, xshift = 1em] {$x \in \SSS$} (controller.180);
    \end{tikzpicture}
    }  
    \caption{
        State transformations with the constraint-lifting function and its inverse used in the control synthesis.
    }
    \label{fig:state_transformation}
\end{figure}
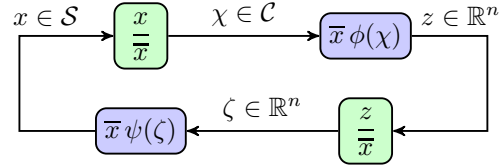

Define the normalized states
\begin{align}
    \chi_{k,1}
        &\isdef
            \dfrac{x_{k,1}}{\overline{x}_1},
    \qquad
    \chi_{k,2}
        \isdef
            \dfrac{x_{k,2}}{\overline{x}_2}.
\end{align}
By construction, the constraint boundaries are mapped to $\pm1$. 
Hence, if $(x_1,x_2)\in\SSS$, then $(\chi_1,\chi_2)\in\SC$.
Next, define the lifted coordinates
\begin{align}
    z_{k,1} 
        &\isdef
            \overline{x}_1\,\phi_1(\chi_{k,1}),
    \qquad
    z_{k,2} 
        \isdef
            \overline{x}_2\,\phi_2(\chi_{k,2}),
    \label{eq:x2_to_z2}
\end{align}
where $\phi_i:\SC\to\BBR$ is the constraint-lifting function.
To simplify controller synthesis, we define the auxiliary variables
\begin{align}
    \zeta_{k,1}
        &\isdef
            \dfrac{z_{k,1}}{\overline{x}_1},
    \qquad
    \zeta_{k,2}
        \isdef
            \dfrac{z_{k,2}}{\overline{x}_2}.
    \label{eq:zeta2_def}
\end{align}
The inverse transformation is then defined by
\begin{align}
    x_{k,1} 
        &=
            \overline{x}_1\,\psi_1(\zeta_{k,1}),
    \qquad
    x_{k,2} 
        =
            \overline{x}_2\,\psi_2(\zeta_{k,2}),
    \label{eq:z2_to_x2}
\end{align}
where $\psi_i:\BBR\to\SSS$.

Under these transformations, the dynamics \eqref{eq:x1_dot}--\eqref{eq:x2_dot} expressed in the lifted coordinates $z$ become
\begin{align}
    z_{k+1,1}
        &=
            \overline{x}_1
            \phi_1
            \bigl(
                \SF_1(z_{k,1},z_{k,2})
            \bigr),
    \label{eq:z1_dot}
    \\
    z_{k+1,2}
        &=
            \overline{x}_2
            \phi_2
            \bigl(
                \SF_2(z_{k,1},z_{k,2},u_k)
            \bigr),
    \label{eq:z2_dot}
\end{align}
where
{\small{
\begin{align}
    \SF_1(z_{k,1},z_{k,2})
        &\isdef
            \dfrac{
                f_1\bigl(\overline{x}_1\psi_1(\zeta_{k,1})\bigr)
            }
            {\overline{x}_1}
            \nn \\ &\quad 
            +
            \dfrac{\overline{x}_2}{\overline{x}_1}
            g_1\bigl(\overline{x}_1\psi_1(\zeta_{k,1})\bigr)
            \psi_2(\zeta_{k,2}),
    \\
    \SF_2(z_{k,1},z_{k,2},u_k)
        &\isdef
            \dfrac{
                f_2\bigl(
                    \overline{x}_1\psi_1(\zeta_{k,1}),
                    \overline{x}_2\psi_2(\zeta_{k,2})
                \bigr)
            }
            {\overline{x}_2}
            \nn \\ &\quad 
            +
            \dfrac{
                g_2\bigl(
                    \overline{x}_1\psi_1(\zeta_{k,1}),
                    \overline{x}_2\psi_2(\zeta_{k,2})
                \bigr)
            }
            {\overline{x}_2}
            u_k.
\end{align}
}}

\section{Controller Synthesis}
\label{sec:Controller_design}

Let $x_{k,1}^{\rmd}\in\SSS$ denote the desired value of $x_1$. Define the corresponding normalized and lifted states
\begin{align}
    \chi_{k,1}^{\rmd}
        &\isdef
            \frac{x_{k,1}^{\rmd}}{\overline{x}_1},
    \ 
    z_{k,1}^{\rmd}
        \isdef
            \overline{x}_1\,\phi_1(\chi_{k,1}^{\rmd}),
    \ 
    \zeta_{k,1}^{\rmd}
        \isdef
            \frac{z_{k,1}^{\rmd}}{\overline{x}_1}.       
    \label{eq:z1d}
\end{align}
and define the tracking error
\begin{align}
    e_{k,1}
        &\isdef
            \psi_1
            \left(
                \zeta_{k,1}
            \right)
            -
            \psi_1
            \left(
                \zeta_{k,1}^{\rmd}
            \right).
    \label{eq:e1_def}
\end{align}

\subsection{\texorpdfstring{$e_1$}{e1} stabilization}

From \eqref{eq:e1_def} and the lifted dynamics \eqref{eq:z1_dot},
\begin{align}
    e_{k+1,1}
        &=
            \SF_1(z_{k,1},z_{k,2})
            -
            \psi_1
            \!\left(
                \zeta_{k+1,1}^{\rmd}
            \right).
    \label{eq:e1_dynamics_initial}
\end{align}

If $z_{k,2}$ could be selected directly, imposing
\begin{align}
    \SF_1(z_{k,1},z_{k,2})
        -
        \psi_1
        \!\left(
            \zeta_{k+1,1}^{\rmd}
        \right)
        =
        \rho_1 e_{k,1},
    \label{eq:first_BS_Step}
\end{align}
where $|\rho_1|<1$, would yield $e_{k+1,1}=\rho_1 e_{k,1}$.
However, $z_{k,2}$ is a state and cannot be selected arbitrarily, so \eqref{eq:first_BS_Step} cannot be enforced directly.

Instead, we define the virtual desired state
\begin{align}
    z_{k,2}^{\rmd}
        &\isdef
            \SF_1^\rmI
            \left(
                z_{k,1},
                \rho_1 e_{k,1}
                +
                \psi_1
                \!\left(
                    \zeta_{k+1,1}^{\rmd}
                \right)
            \right),
    \label{eq:z2dk}
\end{align}
where $\SF_1^\rmI(z_1,\cdot)$ denotes the inverse mapping satisfying
    $z_2=\SF_1^\rmI(z_1,y)
    \Longleftrightarrow 
\SF_1(z_1,z_2)=y .$
By construction,
\begin{align}
    \SF_1(z_{k,1},z_{k,2}^{\rmd})
        -
        \psi_1
        \!\left(
            \zeta_{k+1,1}^{\rmd}
        \right)
        &=
            \rho_1 e_{k,1}.
    \label{eq:first_BS_virtual_condition}
\end{align}

Define $\zeta_{k,2}^{\rmd}\isdef z_{k,2}^{\rmd}/\overline{x}_2$ and the second backstepping error
\begin{align}
    e_{k,2}
        &\isdef
            \psi_2
            \!\left(
                \zeta_{k,2}
            \right)
            -
            \psi_2
            \!\left(
                \zeta_{k,2}^{\rmd}
            \right).
    \label{eq:e2_def}
\end{align}

Adding and subtracting $\SF_1(z_{k,1},z_{k,2}^{\rmd})$ in \eqref{eq:e1_dynamics_initial} and using \eqref{eq:first_BS_virtual_condition} gives
\begin{align}
    e_{k+1,1}
        &=
            \rho_1 e_{k,1}
            +
            \SF_1(z_{k,1},z_{k,2})
            -
            \SF_1(z_{k,1},z_{k,2}^{\rmd}).
    \label{eq:e1_dynamics_difference}
\end{align}
From the definition of $\SF_1$, the remaining difference is
\begin{align}
    \SF_1(z_{k,1},z_{k,2})
        -
        \SF_1(z_{k,1},z_{k,2}^{\rmd})
        &=
            a_k e_{k,2},
\end{align}
where
\begin{align}
    a_k
        \isdef
            \frac{\overline{x}_2}{\overline{x}_1}
            g_1
            \bigl(
                \overline{x}_1\psi_1(\zeta_{k,1})
            \bigr).
    \label{eq:ak_def}
\end{align}
Hence,
\begin{align}
    e_{k+1,1}
        &=
            \rho_1 e_{k,1}
            +
            a_k e_{k,2}.
    \label{eq:e1_actual_dynamics}
\end{align}
Unlike the ideal first backstepping step, the actual first-error dynamics carry a coupling term proportional to $e_{k,2}$, since $z_{k,2}$ generally differs from $z_{k,2}^{\rmd}$.

\subsection{\texorpdfstring{$e_2$}{e2} stabilization}

From \eqref{eq:e2_def} and the lifted dynamics \eqref{eq:z2_dot},
\begin{align}
    e_{k+1,2}
        &=
            \SF_2(z_{k,1},z_{k,2},u_k)
            -
            \psi_2
            \!\left(
                \zeta_{k+1,2}^{\rmd}
            \right).
    \label{eq:e2_dynamics_initial}
\end{align}
We choose the control input to satisfy
\begin{align}
    \SF_2(z_{k,1},z_{k,2},u_k)
        -
        \psi_2
        \!\left(
            \zeta_{k+1,2}^{\rmd}
        \right)
        =
        \rho_{2,k} e_{k,2},
    \label{eq:second_BS_Step}
\end{align}
where $|\rho_{2,k}|<1$. Using the inverse mapping $\SF_2^\rmI$ gives
\begin{align}
    u_k
        =
            \SF_2^\rmI
            \Bigg(
                z_{k,1},
                z_{k,2},
                \rho_{2,k}e_{k,2}
                +
                \psi_2
                \!\left(
                    \zeta_{k+1,2}^{\rmd}
                \right)
            \Bigg),
\label{eq:u_k_constraint_lifting}
\end{align}
where $\SF_2^\rmI(z_1,z_2,\cdot)$ denotes the inverse mapping satisfying
    $u=\SF_2^\rmI(z_1,z_2,y)
    \Longleftrightarrow 
    \SF_2(z_1,z_2,u)=y ,$
and thus
\begin{align}
    e_{k+1,2}
        &=
            \rho_{2,k} e_{k,2}.
    \label{eq:e2_dynamics_theorem}
\end{align}

\begin{remark}
Consequently, if for some $k_\rms\geq0,$ $\rho_{2,k}=0$ for all $k\ge k_\rms,$ then $e_{k,2}=0$ for all $k> k_\rms.$
\end{remark}

\subsection{Stability Analysis}

Combining \eqref{eq:e1_actual_dynamics} and \eqref{eq:e2_dynamics_theorem}, the closed-loop error dynamics are
\begin{align}
    \begin{bmatrix}
        e_{k+1,1}
        \\
        e_{k+1,2}
    \end{bmatrix}
        &=
            \begin{bmatrix}
                \rho_1 & a_k
                \\
                0 & \rho_{2,k}
            \end{bmatrix}
            \begin{bmatrix}
                e_{k,1}
                \\
                e_{k,2}
            \end{bmatrix}.
    \label{eq:closed_loop_error_dynamics}
\end{align}
This is a triangular linear time-varying system.
Its diagonal entries $\rho_1$ and $\rho_{2,k}$ are each individually contractive by hypothesis, and $a_k$ is bounded whenever $g_1$ is bounded on the operating range of $x_{k,1}$.
For a triangular linear time-varying system whose diagonal entries are uniformly bounded away from $1$ in magnitude and whose off-diagonal entry is uniformly bounded, the origin is exponentially stable, since each diagonal subsystem is exponentially stable and the coupling term acts as a bounded-input contribution to an already exponentially stable subsystem \cite[Ch.~24]{rugh1996linear}.
Consequently, if $|\rho_1|<1$ and $\sup_k|\rho_{2,k}|<1$, the origin of \eqref{eq:closed_loop_error_dynamics} is exponentially stable.

\section{Admissibility and Forward Invariance of the Constraint-Lifting Transformations}
\label{sec:Admisibility_constraint_leifting}
Admissibility of the constraint-lifting maps $\phi_1$ and $\phi_2$ requires that their arguments remain within the domain $(-1,1)$. 
From \eqref{eq:z1_dot}-\eqref{eq:z2_dot}, this condition is satisfied if, for all $k \ge 0,$
\begin{align}
    |\SF_1(z_{k,1},z_{k,2})| &< 1,
    \quad
    |\SF_2(z_{k,1},z_{k,2},u_k)| < 1.
\end{align}
These conditions ensure that the arguments of $\phi_1$ and $\phi_2$ remain in $(-1,1)$, so that the transformations \eqref{eq:z1_dot}--\eqref{eq:z2_dot} are well defined along the closed-loop trajectories.
Define the admissible sets
\begin{align}
    \SA_1
        &\isdef
            \left\{
                (z_1,z_2)\in\BBR^2
                :
                |\SF_1(z_1,z_2)| < 1
            \right\}, 
    \\
    \SA_2
        &\isdef
            \left\{
                (z_1,z_2,u)\in\BBR^3
                :
                |\SF_2(z_1,z_2,u)| < 1
            \right\}.
\end{align}

\subsection{Forward Invariance of the Admissible Domain of $\phi_2$}

Enforcing the deadbeat choice $\rho_{2,k}=0$ from the initial time may be overly restrictive, since the resulting control input must satisfy not only the admissibility condition of the lifting function $\phi_2$, but also the coupling into $\phi_1$ identified in Section~\ref{sec:Controller_design}. 
To alleviate this restriction, a transient state-dependent gain $\rho_{2,k}$ with $|\rho_{2,k}|<1$ is employed so that the admissibility condition of $\phi_2$ is preserved while the trajectory evolves toward a region where the deadbeat choice becomes admissible. 
The following result establishes forward invariance of the admissible domain of $\phi_2$ during this transient phase.

\begin{theorem}[Forward Invariance of the Admissible Domain of $\phi_2$]
\label{thm:forward_invariance_phi_2}

Consider the transformed dynamics \eqref{eq:z2_dot} with the control law \eqref{eq:u_k_constraint_lifting}. 
Suppose that, for all $0 \le k < k_\rms$, the gain $\rho_{2,k}$ satisfies $|\rho_{2,k}| < 1$ and 
\begin{align}
    \left|
        \rho_{2,k} e_{k,2}
        +
        \psi_2\!\left(
            \zeta_{k+1,2}^{\rmd}
        \right)
    \right|
    &< 1.
    \label{eq:phi2_admissibility_transient_theorem}
\end{align}
Then, for all $0 \le k < k_\rms$, where $k_\rms$ is the finite switching time, the argument of $\phi_2$ remains admissible. 
Hence, the admissible domain of $\phi_2$ is forward invariant for all $0 \le k < k_\rms.$
\end{theorem}

\begin{proof}
    See Appendix \ref{proof:forward_invariance_phi_2}.
\end{proof}

Theorem \ref{thm:forward_invariance_phi_2} implies that the argument of $\phi_2$ remains in its admissible domain throughout the transient phase $0 \le k < k_\rms$. 
Consequently, the admissible domain of $\phi_2$ is forward invariant along the closed-loop trajectories during this phase.

\subsection{Conditional Deadbeat Property of the Second Backstepping Step}
The transient phase described above ensures that the argument of the lifting function $\phi_2$ remains admissible while the trajectory evolves toward a region where the deadbeat choice becomes feasible. 
Once such a region is reached, the controller can be switched to the deadbeat choice $\rho_{2,k}=0$. 
The following result establishes that, provided the admissibility condition remains satisfied at the switching time, the second backstepping step operates in deadbeat mode thereafter while preserving admissibility of $\phi_2$.

\begin{theorem}[Conditional Deadbeat Property of the Second Backstepping Step]
\label{thm:conditional_deadbeat_phi2}

Consider the transformed dynamics \eqref{eq:z2_dot} with the control law \eqref{eq:u_k_constraint_lifting}. 
Assume that Theorem \ref{thm:forward_invariance_phi_2} is satisfied.
If, at $k=k_\rms$, the deadbeat choice $\rho_{2,k}=0$ satisfies
\begin{align}
    \left|
        \psi_2\!\left(
            \zeta_{k_\rms+1,2}^{\rmd}
        \right)
    \right|
    < 1,
    \label{eq:psi2condition}
\end{align}
and, for all $k \ge k_\rms,$ $\rho_{2,k}=0,$
then, for all $k \ge k_\rms,$
    $e_{k+1,2} = 0,$
and the argument of $\phi_2$ remains admissible for all $k \ge k_\rms$. 
\end{theorem}

\begin{proof}
    See Appendix \ref{proof:conditional_deadbeat_phi2}.
\end{proof}

Theorem \ref{thm:conditional_deadbeat_phi2} implies that, once the switching condition \eqref{eq:psi2condition} is satisfied, the second-step error is driven to zero in one step and remains identically zero thereafter. 
Consequently, the admissibility of the argument of $\phi_2$ is preserved for all $k \ge k_\rms$.

\subsection{Forward Invariance of the Admissible Domain of $\phi_1$}

The following result establishes forward invariance of the admissible domain of the lifting function $\phi_1$ once the second backstepping step has entered deadbeat mode.
Unlike a claim stated for all $k\geq0$, admissibility of $\phi_1$ here is established only from the switching time $k_\rms$ onward: by \eqref{eq:e1_actual_dynamics}, $\SF_1(z_{k,1},z_{k,2})$ depends on $e_{k,2}$ through the coupling term $a_ke_{k,2}$, and $e_{k,2}=0$ is guaranteed by Theorem~\ref{thm:conditional_deadbeat_phi2} only for $k>k_\rms$.

\begin{theorem}[Forward Invariance of the Admissible Domain of $\phi_1$]
\label{thm:forward_invariance_phi_1}
Consider the transformed dynamics \eqref{eq:z1_dot} and the first contraction condition \eqref{eq:first_BS_Step} with $\rho_1=0$.
Assume that $|x_{k,1}^{\rmd}|<\overline{x}_1$ for all $k\geq0$, and that the conditions of Theorem~\ref{thm:conditional_deadbeat_phi2} hold at some switching time $k_\rms\geq0$.
Then, for all $k>k_\rms$,
\begin{align}
    \left|
        \SF_1(z_{k,1},z_{k,2})
    \right|
    <
    1.
    \label{eq:phi1_domain_requirement}
\end{align}
\end{theorem}

\begin{proof}
    See Appendix \ref{proof:forward_invariance_phi_1}.
\end{proof}

Theorem \ref{thm:forward_invariance_phi_1} implies that the argument of $\phi_1$ remains in its admissible domain for all $k > k_\rms$. 
Consequently, the admissible domain of $\phi_1$ is forward invariant along the closed-loop trajectories from the switching time onward.
Theorem~\ref{thm:forward_invariance_phi_1} does not address $0\leq k\leq k_\rms$: during and at the switching step, $e_{k,2}$ need not be zero, and the coupling term $a_ke_{k,2}$ can perturb $\SF_1$ away from $\chi_{k+1,1}^{\rmd}$.
Admissibility of $\phi_1$ over that interval is not established by this framework and is not claimed.

\section{Numerical Example}
\label{sec:Numerical_Example}
This section illustrates the controller developed in Section \ref{sec:Controller_design} and verifies the admissibility conditions for $\phi_1$ and $\phi_2$ established in Section \ref{sec:Admisibility_constraint_leifting}.
Consider the controlled Roup oscillator \cite{roup2001adaptive}, discretized by forward Euler with sampling period $T>0$,
\begin{align}
    x_{k+1,1} &= x_{k,1} + T x_{k,2},
    \label{eq:roup_x1_discrete}
    \\
    x_{k+1,2} 
        &=
            x_{k,2} - 
            T\lambda\!\left(x_{k,1}^2+\frac{x_{k,2}^2}{\omega_0^2}
            -a^2\right)
            \!x_{k,2}
            \nn \\ &\quad 
            - T\omega_0^2 x_{k,1} + Tu_k,
    \label{eq:roup_x2_discrete}
\end{align}
so that $f_1(x_1)=x_1$, $g_1(x_1)=T$, $g_2(x_1,x_2)=T$, and $f_2$ is the drift term in \eqref{eq:roup_x2_discrete}.
The objective is to track the command $x_{k,1}^{\rmd}$ while satisfying $|x_{k,1}|<\overline{x}_1$ and $|x_{k,2}|<\overline{x}_2$.

\subsection{Controller}

In this example, the inverse sigmoid lifting function $\phi(\chi) = \mathrm{atanh}(\chi)$ is used for both states, and $\rho_1 = 0$.
Since $g_1(x_1)=T$ is constant, the virtual desired state from \eqref{eq:z2dk} reduces to the physical-coordinate expression
\begin{align}
    x_{k,2}^{\rmd} &= \frac{x_{k+1,1}^{\rmd}-x_{k,1}}{T}.
    \label{eq:roup_x2d}
\end{align}
The control law \eqref{eq:u_k_constraint_lifting}, specialized to \eqref{eq:roup_x1_discrete}--\eqref{eq:roup_x2_discrete}, is
\begin{align}
    u_k &= \frac{x_{k+1,2}^{\rmd} + \rho_{2,k}\!\left(x_{k,2}-x_{k,2}^{\rmd}\right) - x_{k,2}}{T}
    \nn\\
    &\quad + \lambda\!\left(x_{k,1}^2+\frac{x_{k,2}^2}{\omega_0^2}-a^2\right)\!x_{k,2} + \omega_0^2 x_{k,1}.
    \label{eq:roup_control_switching}
\end{align}

\begin{remark}
    The control law \eqref{eq:roup_control_switching} is expressed entirely in physical coordinates.
The lifting function $\phi$ and its inverse $\psi$ do not appear explicitly, since they cancel algebraically in the derivation leading to \eqref{eq:u_k_constraint_lifting}.
Consequently, $u_k$ remains well defined and well conditioned as the state approaches the constraint boundary, in contrast to the ill-conditioning that direct evaluation of $\phi$ or its derivatives would exhibit near the boundary.
\end{remark}
\subsection{Reference, Admissible Initial Condition, and Switching Law}

Consider the desired trajectory $x_{k,1}^{\rmd}=r\cos(\omega kT)$.
As in Section~\ref{sec:Admisibility_constraint_leifting}, an initial condition must satisfy the one-step position constraint $|x_{0,1}+Tx_{0,2}|<\overline{x}_1$ for $\phi_1$-admissibility to be attainable at all.
This condition alone does not guarantee $\phi_2$-admissibility at $k=0$: the required virtual state $x_{k,2}^{\rmd}$ in \eqref{eq:roup_x2d} depends on $x_{k,1}$ and the reference, and can exceed $\overline{x}_2$ even when the position constraint is satisfied.
Enforcing $\rho_{2,k}=0$ unconditionally from $k=0$ is therefore not always admissible, motivating the transient switching gain used below.

Define the raw second-state error $\Delta_k \isdef x_{k,2}-x_{k,2}^{\rmd}$, so that $e_{k,2}=\Delta_k/\overline{x}_2$.
We choose
\begin{align}
    \rho_{2,k} &=
    \begin{cases}
        0, & |\Delta_k| < \overline{x}_2,
        \\[4pt]
        1-\dfrac{\overline{x}_2}{2|\Delta_k|}, & |\Delta_k| \geq \overline{x}_2.
    \end{cases}
    \label{eq:rho2_roup_switching}
\end{align}
As in earlier constructions of this type, this switching law is not unique; it is one admissible choice consistent with Theorems~\ref{thm:forward_invariance_phi_2} and \ref{thm:conditional_deadbeat_phi2}, verified numerically below.

\subsection{Results}

We simulate \eqref{eq:roup_control_switching} with \eqref{eq:rho2_roup_switching} from the initial condition $x_0=(-1.40,-0.45)$, for which $x_{1,1}=x_{0,1}+Tx_{0,2}=-1.49$ satisfies the position constraint.
Figure~\ref{fig:roup_switching_phase_plane} shows the resulting closed-loop trajectory under \eqref{eq:rho2_roup_switching}.
Figure~\ref{fig:roup_switching_states} shows the corresponding state, control input, and gain histories.
At this initial condition, the deadbeat choice $\rho_{2,k}=0$ from $k=0$ is inadmissible.
It drives $x_{1,2}$ to $11.13$, which exceeds $\overline{x}_2=1.5$.
The gain $\rho_{2,k}$ starts at $0.93$ and decreases as the trajectory approaches feasibility, switching to $\rho_{2,k}=0$ at $k_\rms=14$ and remaining there; $x_1$ and $x_2$ stay strictly within their bounds throughout, and $x_1$ converges to the commanded value after the switch.

\begin{figure}[!ht]
    \centering
    \includegraphics[width=\columnwidth]{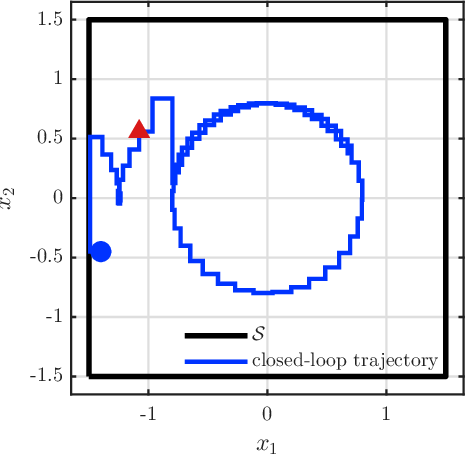}
    \caption{Closed-loop phase portrait for the Roup oscillator under the transient switching gain \eqref{eq:rho2_roup_switching}, starting from an initial condition for which the deadbeat choice $\rho_{2,k}=0$ is inadmissible at $k=0$. The circle marks $x_0$; the triangle marks the state at the switching time $k_\rms$.}
    \label{fig:roup_switching_phase_plane}
\end{figure}

\begin{figure}[!ht]
    \centering
    \includegraphics[width=\columnwidth]{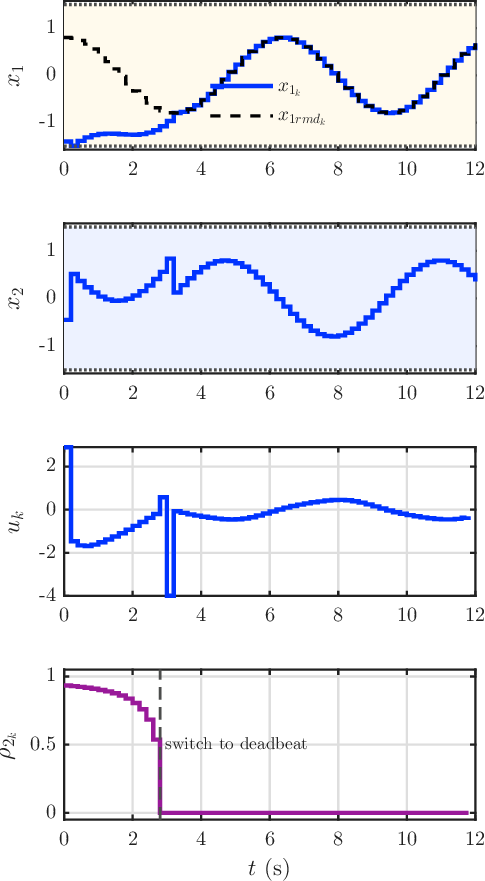}
    \caption{State, control input, and gain evolution for the Roup oscillator under the transient switching gain \eqref{eq:rho2_roup_switching}. The safe sets for $x_1$ and $x_2$ are shown in shaded yellow and blue, respectively. The gain $\rho_{2,k}$ decreases from the initial transient and switches to $0$ at $k_\rms$, indicated by the dashed vertical line.}
    \label{fig:roup_switching_states}
\end{figure}

\section{Conclusions}

This paper developed a constraint-lifting control framework for discrete-time strict-feedback nonlinear systems subject to state constraints. 
The proposed approach transforms the constrained control problem into an equivalent unconstrained one using smooth constraint-lifting mappings constructed from strictly increasing sigmoid functions. 
Controller synthesis is performed in the lifted coordinates using a recursive backstepping design, while the inverse transformation guarantees that the original state constraints are satisfied.

Sufficient conditions on the controller gains were derived to ensure asymptotic stability of the closed-loop system and forward invariance of the admissible domain of the lifting functions. 
The analysis further established a conditional deadbeat property for the second backstepping step once the closed-loop trajectory enters a region where the admissibility conditions are satisfied. 
Numerical simulations on a constrained Roup oscillator demonstrated that the proposed method preserves state constraints while achieving stable regulation and tracking.


\label{sec:conclusions}

\appendix
\section{Proofs}

\subsection{Proof of Theorem \ref{thm:forward_invariance_phi_2}}
\label{proof:forward_invariance_phi_2}
\begin{proof}
It follows from \eqref{eq:second_BS_Step} that 
\begin{align}
    \SF_2(z_{k,1},z_{k,2},u_k)
        &=
            \rho_{2,k} e_{k,2}
            +
            \psi_2
            \left(
                \zeta_{k+1,2}^{\rmd}
            \right). 
    \label{eq:SF2_relation_theorem}
\end{align}
Hence, \eqref{eq:phi2_admissibility_transient_theorem} implies, for all $0 \leq k < k_\rms,$
    $\left|
        \SF_2(z_{k,1},z_{k,2},u_k)
    \right|
        <
            1.$ 
Therefore, the argument of $\phi_2$ remains within its admissible domain for all
$0 \leq k < k_\rms$.

Moreover, the second-step error dynamics satisfies \eqref{eq:e2_dynamics_theorem}.
Since $|\rho_{2,k}|<1$ for all $0 \leq k < k_\rms$, the error dynamics are
contractive during the transient phase. Hence, the admissibility condition of
$\phi_2$ is preserved throughout the transient phase.
\end{proof}

\subsection{Proof of Theorem \ref{thm:conditional_deadbeat_phi2}}
\label{proof:conditional_deadbeat_phi2}

\begin{proof}

Assume that the conditions of Theorem
\ref{thm:forward_invariance_phi_2} hold for all $0 \leq k < k_\rms$.
Suppose that, at time $k_\rms$, the closed-loop trajectory reaches a point for
which the deadbeat choice $\rho_{2,k}=0$ preserves the admissibility of the
argument of $\phi_2$.
Setting $\rho_{2,k}=0$ into \eqref{eq:second_BS_Step} yields
    $\SF_2(z_{k,1},z_{k,2},u_k)
        =
            \psi_2
            \left(
                \zeta_{k+1,2}^{\rmd}
            \right).$
Hence, by \eqref{eq:psi2condition}, the admissibility of the argument of $\phi_2$ is preserved for all $k \geq k_\rms$.

Furthermore, substituting $\rho_{2,k}=0$ into the error dynamics
\eqref{eq:e2_dynamics_theorem} implies that, for all, $k \geq k_\rms,$
    $e_{k+1,2}
        =
            0.$
Therefore, once the switching condition is satisfied, the second-step error is
driven to zero in one step and remains identically zero thereafter. This
establishes the deadbeat property of the second backstepping step.
\end{proof}

\subsection{Proof of Theorem \ref{thm:forward_invariance_phi_1}}
\label{proof:forward_invariance_phi_1}
\begin{proof}

From \eqref{eq:e1_actual_dynamics} with $\rho_1=0$,
\begin{align}
    \SF_1(z_{k,1},z_{k,2})
        &=
            \psi_1
            \left(
                \zeta_{k+1,1}^{\rmd}
            \right)
            +
            a_k e_{k,2}.
    \label{eq:SF1_with_coupling}
\end{align}
By Theorem~\ref{thm:conditional_deadbeat_phi2}, $e_{k,2}=0$ for all $k>k_\rms$.
Hence, for $k>k_\rms$, \eqref{eq:SF1_with_coupling} reduces to
\begin{align}
    \SF_1(z_{k,1},z_{k,2})
        =
            \psi_1
            \left(
                \zeta_{k+1,1}^{\rmd}
            \right)
        =
            \chi_{k+1,1}^{\rmd}.
\end{align}

Since $|x_{k+1,1}^{\rmd}| < \overline{x}_1$ for all $k\geq0$,
it follows that $|\chi_{k+1,1}^{\rmd}| < 1$,
which proves \eqref{eq:phi1_domain_requirement} for all $k>k_\rms$.
\end{proof}

\printbibliography

@phdthesis{portella2026constrained,
  title  = {Adaptive and Safe Control of Nonlinear Systems with Multiplicative Uncertainty and State Constraints},
  author = {Portella Delgado, Jhon Manuel},
  year   = {2026},
  school = {University of Maryland, Baltimore County},
  type   = {{PhD} thesis},
  note   = {Available at \url{http://hdl.handle.net/11603/43179}}
}

@article{Blanchini1999,
  author       = {Franco Blanchini},
  title        = {Set Invariance in Control},
  journal      = {Automatica},
  volume       = {35},
  number       = {11},
  pages        = {1747--1767},
  year         = {1999},
  doi          = {10.1016/S0005-1098(99)00113-2}
}

@article{Gilbert1995,
  author       = {Eugene G. Gilbert and Ilya Kolmanovsky and Kok Tin Tan},
  title        = {Discrete-Time Reference Governors and the Nonlinear Control of Systems with State and Control Constraints},
  journal      = {International Journal of Robust and Nonlinear Control},
  volume       = {5},
  number       = {5},
  pages        = {487--504},
  year         = {1995},
  doi          = {10.1002/rnc.4590050508}
}

@article{Garone2017,
  author       = {Emanuele Garone and Stefano Di Cairano and Ilya Kolmanovsky},
  title        = {Reference and Command Governors for Systems with Constraints: A Survey on Theory and Applications},
  journal      = {Automatica},
  volume       = {75},
  pages        = {306--328},
  year         = {2017},
  doi          = {10.1016/j.automatica.2016.08.013}
}

@article{Tee2009,
  author       = {Keng Peng Tee and Shuzhi Sam Ge and Eng Hock Tay},
  title        = {Barrier Lyapunov Functions for the Control of Output-Constrained Nonlinear Systems},
  journal      = {Automatica},
  volume       = {45},
  number       = {4},
  pages        = {918--927},
  year         = {2009},
  doi          = {10.1016/j.automatica.2008.11.017}
}

@inproceedings{Agrawal2017,
  author       = {Ayush Agrawal and Koushil Sreenath},
  title        = {Discrete Control Barrier Functions for Safety-Critical Control of Discrete Systems with Application to Bipedal Robot Navigation},
  booktitle    = {Proceedings of Robotics: Science and Systems},
  address      = {Cambridge, Massachusetts},
  month        = jul,
  year         = {2017},
  doi          = {10.15607/RSS.2017.XIII.073}
}

@article{Guo2014,
  author       = {Tao Guo and Xiaowei Wu},
  title        = {Backstepping Control for Output-Constrained Nonlinear Systems Based on Nonlinear Mapping},
  journal      = {Neural Computing and Applications},
  volume       = {25},
  pages        = {1665--1674},
  year         = {2014},
  doi          = {10.1007/s00521-014-1650-9}
}

@misc{delgado2026stateconstrained,
  title = {State-Constrained Control of Discrete-Time Nonlinear Systems via Constraint Lifting},
  author = {Delgado, Jhon Manuel Portella and Goel, Ankit},
  year = {2026},
  month = apr,
  doi = {10.48550/arxiv.2604.24989},
  eprint = {2604.24989},
  archivePrefix = {arXiv}
}

@book{rugh1996linear,
  author    = {Rugh, Wilson J.},
  title     = {Linear System Theory},
  edition   = {2nd},
  publisher = {Prentice Hall},
  address   = {Upper Saddle River, NJ},
  year      = {1996},
  series    = {Prentice Hall Information and System Sciences Series},
  isbn      = {0134412052}
}

@article{roup2001adaptive,
  title = {Adaptive stabilization of a class of nonlinear systems with nonparametric uncertainty},
  author = {Roup, A.V. and Bernstein, D.S.},
  journal = {IEEE Transactions on Automatic Control},
  year = {2001},
  month = jan,
  volume = {46},
  number = {11},
  pages = {1821-1825},
  publisher = {Institute of Electrical and Electronics Engineers (IEEE)},
  doi = {10.1109/9.964699}
}
\end{document}